\documentclass[letterpaper,article,oneside,11pt]{memoir}

\usepackage{amsmath}
\usepackage{amssymb}
\usepackage{amsthm}
\usepackage{mathrsfs}
\usepackage{lmodern}
\usepackage[T1]{fontenc}
\usepackage[utf8]{inputenc}
\usepackage{enumitem}
\usepackage{tikz-cd}
\usepackage{xcolor}

\usepackage{tikz-cd}
\usetikzlibrary{cd}
\usepackage{pgfplots}
\pgfplotsset{compat=1.5.1}

\settrims{0pt}{0pt}
\setheadfoot{0.5in}{0.5in}
\setulmarginsandblock{1.0in}{1.0in}{*}
\setlrmarginsandblock{1.0in}{1.0in}{*}
\checkandfixthelayout

\usepackage[backend=biber,style=alphabetic,giveninits,url=false,doi=false,maxbibnames=50,maxcitenames=50,maxnames=50]{biblatex}
\defbibheading{bibliography}[Bibliography]{\section*{#1}\markboth{#1}{#1}}

\counterwithout{section}{chapter}
\counterwithout{equation}{chapter}
\numberwithin{equation}{section}

\renewcommand{\maketitle}{
 {\centering\textbf{\thetitle}\par
  \vspace{0.5em}
  \centering{\theauthor}\par 
  \vspace{2em}
 }
}

\setsecheadstyle{\bfseries}

\setsecnumdepth{subsubsection}
\setsecnumformat{\csname the#1\endcsname. }
\setsubsecheadstyle{\normalfont\bfseries}
\setaftersubsecskip{-0.5em}
\setsubsubsecheadstyle{\normalfont\bfseries}
\setaftersubsubsecskip{-0.5em}

\OnehalfSpacing
\newtheorem{theorem}[equation]{Theorem}

\newtheorem{proposition}[equation]{Proposition}
\newtheorem{corollary}[equation]{Corollary}
\newtheorem{lemma}[equation]{Lemma}

\theoremstyle{definition}
\newtheorem{example}[equation]{Example}

\newtheorem{definition}[equation]{Definition}

\newcommand{\Holder}{H\"{o}lder}

\newcommand{\R}{\mathbb{R}}
\newcommand{\N}{\mathbb{N}}

\newcommand{\metric}{\mathsf{d}}
\newcommand{\nbar}{|\!|}
\newcommand{\intd}{\,\mathrm{d}}

\newcommand{\entropy}{\mathsf{h}}
\newcommand{\press}{\mathsf{P}}

\newcommand{\define}[1]{\textbf{#1}}
\renewcommand{\emptyset}{\varnothing}
\newcommand{\floor}[1]{\lfloor {#1} \rfloor}

\renewcommand{\intd}{\,\mathsf{d}}
\newcommand{\vertex}[1]{[#1]}
\newcommand{\svtx}[1]{\scriptstyle{[#1]}}

\newcommand{\ceil}[1]{\lceil #1 \rceil}
\newcommand{\foll}{\mathsf{F}}

\newcommand{\lang}{\mathcal{L}}
\newcommand{\good}{\mathcal{G}}
\newcommand{\suff}{\mathcal{S}}
\renewcommand{\pref}{\mathcal{P}}

\newcommand{\diag}{\mathsf{D}}

\newcommand{\vtx}{\mathsf{vtx}}
\newcommand{\pth}{\mathsf{pth}}

\title{Unique equilibrium states for some intermediate beta transformations}
\author{Leonard Carapezza, Marco López and Donald Robertson}

\begin{document}

\maketitle

\begin{abstract}
We prove uniqueness of equilibrium states for subshifts corresponding to intermediate beta transformations with $\beta > 2$ having the property that the orbit of 0 is bounded away from 1.
\end{abstract}

\section{Introduction}
\label{sec:introduction}

A \define{topological dynamical system} is any pair $(X,T)$ where $X$ is a compact metric space and $T : X \to X$ is a continuous map.
Often there are many Borel probability measures on $X$ that are $T$ invariant and ergodic.
One way to choose a specific $T$ invariant probability measure with which to analyze the statistics of a system is to attempt to maximize the measure-theoretic entropy $\nu \mapsto \entropy(T,\nu)$.
More generally, given $\phi : X \to \R$ measurable one can choose an invariant measure by attempting to maximize
\begin{equation}
\label{eqn:freeEnergy}
\nu \mapsto \int \phi \intd \nu + \entropy(T,\nu)
\end{equation}
over all $T$ invariant probabilities.
When $T$ is expansive and $\phi$ is continuous this can always be done in at least one way~\cite[Page~224]{MR648108}.
Measures maximizing \eqref{eqn:freeEnergy} are called \define{equilibrium states} of $T$ for the function $\phi$.

Here we are interested in uniqueness of equilibrium states for a class of topological systems coming from generalizations of the beta transformations $F(x) = \beta x \bmod 1$ on $[0,1)$ for $\beta > 1$.
Given $0 \le \alpha < 1$ and $\beta > 1$ we consider the map $F : [0,1) \to [0,1)$ defined by
\begin{equation}
\label{eqn:alphaBetaMap}
F(x) = \beta x + \alpha - \floor{\beta x + \alpha} = \beta x + \alpha \bmod 1
\end{equation}
for all $0 \le x < 1$.
Such maps are called \define{alpha--beta transformations} or \define{intermediate beta transformations} and have been studied extensively~\cite{MR0097374,MR0142719,MR0466492,MR1027793,MR1452189,MR3046278,2017arXiv170908035L,faller_point_2009}.
In this paper we focus on intermediate beta transformations with $\beta > 2$.

Put $\ell = \ceil{\alpha + \beta} - 1$.
Using the partition
\begin{equation}
\label{eqn:alphaBetaPartition}
J_0 = \left[ 0, \frac{1-\alpha}{\beta} \right),
J_1 = \left[ \frac{1-\alpha}{\beta}, \frac{2-\alpha}{\beta} \right),
\dots,
J_\ell = \left[ \frac{\ell - \alpha}{\beta}, 1 \right)
\end{equation}
of $[0,1)$ one can associate to any intermediate beta transformation $F$ a subshift $\Sigma$ of $\{ 0,\dots, \ell \}^\N$.
We call such subshifts \define{intermediate beta shifts}.
To describe our main result denote by $a$ and $b$ respectively the infimum and supremum of $\Sigma$ with respect to the lexicographic order on $\{ 0, \dots, \ell \}^\N$.
We will make repeated reference to the sets
\begin{align*}
\diag(a) &= \{ n \in \N : (b_1,\dots,b_n) = (a_j,\dots,a_{j+n-1}) \textup{ for some } j \in \N \} \\
\diag(b) &= \{ n \in \N : (a_1,\dots,a_n) = (b_j,\dots,b_{j+n-1}) \textup{ for some } j \in \N \}
\end{align*}
and will be particularly interested in whether they are bounded.
Our main result is the following.

\begin{theorem}
\label{thm:mainTheorem}
Fix $0 \le \alpha < 1$ and $\beta > 2$ and $\phi : \Sigma \to \R$ a \Holder{} continuous function.
If the set $\diag(a)$ is bounded then $T$ has a unique equilibrium state for the function $\phi$.
\end{theorem}

The hypothesis $\beta > 2$ is related to transitivity.
One says that $F$ is transitive if every open subinterval $J \subset [0,1)$ has associated to it some $m \in \N$ such that the sets $J,F(J),\dots,F^m(J)$ cover $[0,1)$.
The parameters $\alpha,\beta$ for which the intermediate beta transformation is transitive have been determined: in \cite{palmer_thesis,MR1027793} it is shown for $\alpha + \beta < 2$ that non-transitivity takes place in a countable union of connected regions in the $(\alpha,\beta)$ plane.
In particular, for $\alpha + \beta < 2$ the condition $\beta > \sqrt{2}$ implies transitivity.
For $\alpha + \beta \ge 2$ and $\beta < 2$ there is a further non-transitive region bounded by the curves $\alpha\beta = 1$ and $\beta(\alpha + \beta - 2) = \beta-1$.
The assumption $\beta > 2$ avoids all these regions and is sufficient for our goal (see Proposition~\ref{prop:rts}).
See \cite{MR602242} for related work on the support of the unique measure of maximal entropy.

Equilibrium states have long been a topic of interest and their uniqueness and statistical properties have been studied in various settings.
We briefly mention some of the previous work done in similar settings.
In \cite{MR1074766} the authors establish uniqueness for a collection of interval maps that includes intermediate beta transformations but require the additional assumption on the potential $\phi$ that $\sup \phi < \press(\phi)$. This condition is referred to as ``small potentials'' in \cite{buzzi_number_1995}, where the author also relaxes the hypothesis of the main result in \cite{MR1074766} to include maps with homtervals and discontinuous potentials.

The consequences of having unique equilibrium states have been studied by many authors.
In \cite{MR1837214, MR1994883} multifractal results are given for functions coming from a subspace for which equilibrium states are unique.
If there exists a dense subspace of continuous functions with unique equilibrium states, then \cite{MR3001770} gives  multifractal results that apply to all continuous functions, \cite{MR3568724} shows that every invariant measure can be arbitrarily well approximated in the weak star topology and in entropy by measures that are unique equilibrium states, \cite{MR3628919} shows that the graph of the entropy map restricted to ergodic measures is dense in the graph of the entropy map, and \cite{MR1025756,MR3568724,MR3628919} show that large deviation principles are obtained.
The space of \Holder{} continuous functions on $\Sigma$ is dense in the space of all continuous functions on $\Sigma$ so these results apply in particular to intermediate beta transformations.
We refer the reader to \cite[Section~4]{MR3046278} for an explanation of how, mutatis mutandis, these results for intermediate beta shifts imply results for intermediate beta transformations.

Our proof of Theorem~\ref{thm:mainTheorem} makes use of \cite[Theorem~C]{MR3046278}, which gives some sufficient conditions for a subshift to have unique equilibrium states with respect to a potential.
This is discussed further in Section~\ref{sec:preliminaries}.

If the subshift $\Sigma$ has specification (see Section~\ref{subsec:spec} for the definition) then the conclusion of Theorem~\ref{thm:mainTheorem} is an immediate consequence of \cite{MR0399413}.
We give in Section~\ref{sec:examples} an example of $0 \le \alpha < 1$ and $\beta > 2$ such that the associated subshift $\Sigma$ satisfies the hypothesis of Theorem~\ref{thm:mainTheorem} but does not have specification.
In doing so the following characterization of specification for intermediate beta transformations with $\beta > 2$, proved in Section~\ref{sec:intBetaTransformation}, will be used.

\begin{theorem}
\label{thm:specChar}
Fix $0 \le \alpha < 1$ and $\beta > 2$.
The subshift $\Sigma$ has specification if and only if the sets $\diag(a)$ and $\diag(b)$ are both bounded.
\end{theorem}

We also prove in Section~\ref{sec:examples} that, in a certain sense, the hypothesis of Theorem~\ref{thm:mainTheorem} is often true.

\begin{theorem}
\label{thm:denseness}
For every $\beta > 2$ there is a dense set of $0 \le \alpha < 1$ such that $\diag(a)$ is bounded.
\end{theorem}

In the case where $\alpha=0$ the subshift $\Sigma$ has specification if and only if the set $\diag(b)$ is bounded. This result, as well as results regarding the size of the set of parameters $\beta$ such that the corresponding shift has specification (or other properties) are in \cite{MR1452189}. 

The remainder of the paper runs as follows.
We present in Section~\ref{sec:preliminaries} some background on subshifts and topological pressure.
In Section~\ref{sec:intBetaTransformation} we describe subshifts associated to intermediate beta transformations in terms of infinite graphs and prove Theorem~\ref{thm:specChar}.
The proof of Theorem~\ref{thm:mainTheorem} is given in Section~\ref{sec:proof}.
Finally, in Section~\ref{sec:examples}, we prove Theorem~\ref{thm:denseness} and give an example showing that the hypothesis of Theorem~\ref{thm:mainTheorem} is broader than specification.

This project began under the supervision of Vaughn Climenhaga at the 2017 AMS Mathematics Research Community ``Dynamical Systems: Smooth, Symbolic, and Measurable''.
We would like to thank the AMS for the opportunity to attend the MRC, and to thank Jon Chaika, Vaughn Climenhaga and Daniel Thompson for many useful and engaging conversations related to this project both during and after the MRC.
We would also like to thank the AMS for supporting travel costs after the MRC, which allowed the authors to collaborate in person at the 2018 Joint Mathematics Meeting and at the University of Utah.
M. López would like thank Tony Samuel for helpful discussions regarding intermediate beta transformations.
D.\ Robertson was supported by NSF grant DMS-1703597 during the preparation of this work.

\section{Preliminaries}
\label{sec:preliminaries}

We recall in this section some standard material on subshifts and pressure that we will make use of below.
More details can be found in \cite{MR648108}.

\subsection{Subshifts}

Equip $\{0,\dots,\ell\}^\N$ with the product topology and the standard metric $\metric$ defined by
\[
\metric(x,y) = \inf \left\{ \frac{1}{2^j} : (x_1,\dots,x_j) = (y_1,\dots,y_j) \right\} \cup \{ 1 \}
\]
for all $x, y \in \{0,\dots,\ell\}^\N$.
The \define{lexicographic order} on $\{0,\dots,\ell\}^\N$ is denoted $\preceq$ and defined for all $x,y$ therein by $x \preceq y$ if and only if there is $k \in \N$ with $x_1 = y_1,\dots,x_k = y_k$ and $x_{k+1} \le y_{k+1}$.
The product, metric and lexicographic topologies on $\{0,\dots,\ell\}^\N$ all agree.

Write $\sigma$ for the shift map $(\sigma x)_n = x_{n+1}$ on $\{0,\dots,\ell \}^\N$.
A closed set $\Sigma \subset \{ 0, \dots, \ell \}^\N$ is a \define{subshift} if $\sigma(\Sigma) \subset \Sigma$.
A sequence $w : \{1,\dots,n\} \to \{ 0,\dots,\ell\}$ is a \define{word} of a subshift $\Sigma$ if there is $x \in \Sigma$ with $x_i = w_i$ for all $1 \le i \le n$.
The \define{length} of a word $w$ is the cardinality of its domain and denoted $|w|$.

The \define{language} of a subshift $\Sigma$ is the set $\lang$ of all its words.
Every word $w \in \lang$ defines the \define{cylinder}
\[
[w] = \{ x \in \Sigma : (x_1,\dots,x_n) = (w_1,\dots,w_n) \}
\]
of sequences that begin in agreement with $w$.
For any word $w$ in $\lang$ and any $1 \le i \le j \le |w|$ we write both $w_i^j$ and $(w_i,\dots,w_j)$ for the word obtained by restricting $w$ to $\{i,\dots,j\}$.
Similarly, for any $1 \le i \le j$ denote by $x_i^j$ and $(x_i,\dots,x_j)$ the word obtained by restricting any sequence $x$ to $\{i,\dots,j\}$. 
Words can be concatenated as follows: given words $w,v$ denote by $wv$ the word $(w_1,\dots,w_{|w|},v_1,\dots,v_{|v|})$.
If $\mathcal{W},\mathcal{V}$ are subsets of $\lang$ write
\[
\mathcal{W} \mathcal{V} = \{ wv \in \lang : w \in \mathcal{W}, v \in \mathcal{V} \}
\]
for the set of all words in $\lang$ that can be written as a word from $\mathcal{W}$ concatenated with a word from $\mathcal{V}$.
Lastly, given $\mathcal{H} \subset \lang$ write $\mathcal{H}_n$ for the set of words of length $n$ in $\mathcal{H}$.

\subsection{Pressure}

Fix a subshift $\Sigma$.
Given $\phi : \Sigma \to \R$ and $n \in \N$ define
\[
(S_n \phi)(x) = \sum_{i=0}^{n-1} \phi(T^i x)
\]
for all $x \in \Sigma$.
We define
\[
\phi(w) = \sup \{ (S_{|w|} \phi)(x) : x \in [w] \}
\]
for all words $w$ in $\lang$.
Given $\mathcal{H} \subset \mathcal{L}$ the quantity
\[
\press(\phi,\mathcal{H})
=
\limsup_{n \to \infty} \frac{1}{n} \log \Lambda_n(\phi,\mathcal{H})
\]
is the \define{topological pressure} along $\phi$ of $\mathcal{H}$ where
\[
\Lambda_n(\phi,\mathcal{H})
=
\sum_{w \in \mathcal{H}_n} \exp( \phi(w) )
\]
for all $n \in \N$.
If $\phi = 0$ then $\Lambda_n(\phi,\mathcal{H})$ is the cardinality of $\mathcal{H}_n$ and $\press(0,\mathcal{H})$ is the exponential growth rate of $n \mapsto |\mathcal{H}_n|$.
In particular $\press(0,\lang)$ is the topological entropy of $\Sigma$.

When $\phi$ is continuous the variational principle~\cite[Theorem~9.10]{MR648108} states that
\begin{equation}
\label{eqn:variational}
\press(\phi) = \sup \left\{ \int \phi \intd \nu + \entropy(\sigma,\nu) : \nu \textup{ a } \sigma \textup{ invariant Borel probability on } \Sigma \right\}
\end{equation}
where $\entropy(\sigma,\nu)$ is the measure-theoretic entropy of $\nu$.
The function $\phi$ is referred to as a \define{potential}.
We will primarily be interested in potentials with the Bowen property.

\begin{definition}
A function $\phi : \Sigma \to \R$ has the \define{Bowen} property if there is $V > 0$ such that, for every $n \in \N$ and any $w \in \lang_n$, one has $|(S_n \phi)(x) - (S_n \phi)(y)| \le V$ for all $x,y \in [w]$.
\end{definition}

Every \Holder{} continuous function $\phi : \Sigma \to \R$ has the Bowen property~\cite[Section~2.2]{MR3046278}.
The following lemma will be useful later.

\begin{lemma}
\label{lem:sup}
Fix $x \in \Sigma$ and let $\mathcal{D} = \{ (x_1,\dots,x_n) : n \in \N \} \subset \lang$.
If $\phi : \Sigma \to \R$ has the Bowen property, then
\[
\press(\phi,\mathcal{D}) = \limsup_{n \to \infty}\frac{1}{n} (S_n\phi)(x)
\]
holds.
\end{lemma}
\begin{proof}
Put $w_n = (x_1,\dots,x_n)$.
Since $\phi$ has the Bowen property there is $V > 0$ such that
\[
|\phi(w_n) - (S_n \phi)(x)| \le V
\]
for all $n \in \N$.
Consequently
\[
\frac{1}{n} (S_n \phi)(x) - \frac{V}{n}
\le
\frac{1}{n} \phi(w_n)
\le
\frac{1}{n} (S_n \phi)(x) + \frac{V}{n}
\]
and, since $\phi(w_n) = \log \Lambda_n(\phi,\mathcal{D})$, applying the limsup gives the result.
\end{proof}

\subsection{Specification}
\label{subsec:spec}

Fix a subshift $\Sigma$.
A set $\good \subset \lang$ has \define{specification} when there is $\tau \in \N$ such that, for any words $w^0,\dots,w^n$ in $\good$ one can find words $v^1,\dots,v^n$ in $\lang_\tau$ such that $w^0 v^1 w^1 \cdots w^{n-1} v^n w^n$ belongs to $\lang$.

\begin{proposition}[cf.\ {\cite[Theorem 6.1]{cft18}}]\label{prop:positiveentropy}
Let $\Sigma$ be subshift having positive diameter with respect to $\metric$ such that $\lang$ has specification.
If $\phi : \Sigma \to \R$ is a continuous function with the Bowen property then every equilibrium state for $\phi$ has positive entropy.
\end{proposition}

\begin{corollary}
\label{cor:positiveentropy}
Let $\Sigma$ be a subshift with positive $\metric$ diameter such that language $\lang$ has specification.
Let $\phi : \Sigma \to \R$ be a continuous function with the Bowen property.
We have $\press(\phi,\mathcal{D}) < \press(\phi,\lang)$ for every $x \in \Sigma$ where $\mathcal{D} = \{ (x_1,\dots,x_n) : n \in \N \}$.
\end{corollary}
\begin{proof}
Fix $x \in \Sigma$.
By Lemma~\ref{lem:sup} we have
\[
\press(\phi,\mathcal{D}) = \limsup_{n \to \infty} \frac{1}{n} (S_n\phi)(x)
\]
and therefore
\[
\press(\phi,\mathcal{D}) = \lim_{k \to \infty} \frac{1}{n_k}(S_{n_k} \phi)(x)
\]
for some $n_k \nearrow \infty$ in $\N$.
By passing to a further subsequence if necessary we get
\[
\press(\phi,\mathcal{D}) = \int \phi \intd \nu
\]
for some $\sigma$ invariant probability $\nu$ on $\Sigma$.
If $\nu$ is not an equilibrium state then $\press(\phi,\mathcal{D}) < \press(\phi,\lang)$ by the variational principle~\eqref{eqn:variational}.
On the other hand, if $\nu$ is an equilibrium state then $\press(\phi,\mathcal{D}) < \press(\phi,\lang)$ by Proposition~\ref{prop:positiveentropy}.
\end{proof}

\section{Intermediate beta transformations}
\label{sec:intBetaTransformation}

\subsection{Lexicographic description}

Fix $0 \le \alpha < 1$ and $\beta > 1$ and let $F : [0,1) \to [0,1)$ be the map $F(x) = \beta x + \alpha \bmod 1$.
Put $\ell = \ceil{\alpha + \beta} - 1$.
As mentioned in Section~\ref{sec:introduction} we associate to every $x \in [0,1)$ the sequence $\Omega(x) \in \{0,\dots,\ell\}^\N$ defined by $\Omega(x)_n = i$ if and only if $F^{n-1}(x) \in J_i$ for all $n \in \N$, where the intervals $J_i$ are those of \eqref{eqn:alphaBetaPartition}.
The subshift $\Sigma$ associated to the pair $(\alpha,\beta)$ is the closure of the image of $\Omega$.

The map $\Omega$ intertwines the linear order on $[0,1)$ with the lexicographic order $\preceq$ on $\Sigma$.
This leads to the following characterization of sequences in $\Sigma$. 
Throughout we write $\lang$ for the language of $\Sigma$ as well as $a = \Omega(0)$ and $b = \lim\limits_{x \nearrow 1} \Omega(x)$.

\begin{lemma}
[{\cite[Theorem~2]{MR570882}}]
\label{lem:lexshift}
A sequence $x \in \{0,\dots,\ell\}^\N$ is in $\Sigma$ if and only if
\begin{equation}
\label{eqn:sequenceorder}
a \preceq \sigma^{k-1}(x) \preceq b
\end{equation}
for all $k \in \N$.
\end{lemma}

\begin{corollary}
\label{cor:lexword}
A word $w : \{1,\dots,n\} \to \{0,\dots,\ell\}$ is in $\lang_n$ if and only if
\begin{equation}
\label{eqn:wordorder}
(a_1,\dots,a_{n-k+1}) \preceq (w_k,\dots,w_n) \preceq (b_1,\dots,b_{n-k+1})
\end{equation}
for all $k \in \{1,\dots,n\}$.
\end{corollary}
\begin{proof}
If $(w_k,\dots,w_n) \prec (a_1,\dots,a_{n-k+1})$, then every $x \in [w]$ satisfies $\sigma^{k-1}(x) \prec a$ and is therefore not in $\Sigma$ by Lemma \ref{lem:lexshift}.
Thus $w$ is not in $\lang_n$.
Similarly, if $(b_1,\dots,b_{n-k+1}) \prec (w_k,\dots,w_n)$ then $w$ is not in $\lang_n$.

Now suppose \eqref{eqn:wordorder} holds for some word $w$.
Fix $i$ maximal with $(w_{n-i+1},\dots,w_n) = (a_1,\dots,a_i)$ and put $x = w\sigma^i(a)$.
By Lemma \ref{lem:lexshift}, it suffices to show that \eqref{eqn:sequenceorder} holds.
Since $\sigma^{n-i}(x) = a$ it suffices to check \eqref{eqn:sequenceorder} for all $k < n-i$.  

Fix $k < n-i$.
We have $(a_1,\dots,a_{n-k}) \preceq (w_{k+1},\dots,w_n)$ by \eqref{eqn:wordorder}.
If $k < n-i$ then maximality of $i$ implies $(a_1,\dots,a_{n-k}) \prec (w_{k+1},\dots,w_n)$.
The initial subword of $\sigma^{k}(x)$ is $w_{k+1}^n$ so we have $a \prec \sigma^{k}(x)$.
For the second inequality in \eqref{eqn:sequenceorder} observe that $(w_{k+1},\dots,w_{n-i}) \preceq (b_1,\dots,b_{n-k-i})$ by \eqref{eqn:wordorder}.
Thus, for $k<n-i$ we have
\[
\sigma^{k}(x)
=
w_{k+1}^{n-i} a
\le
b_1^{n-k-i} a
\le
b_1^{n-k-i} \sigma^{n-k-i}(b)
=
b
\]
with the second inequality holding because $b$ satisfies \eqref{eqn:sequenceorder}. 
\end{proof}

\subsection{Graphical description}
\label{subsec:graph}

Using the lexicographic description of $\lang$ in the previous subsection we present here a graph with directed edges labeled by $\{0,\dots,\ell\}$ and a root vertex with the following property: that a word belongs to $\lang$ if and only if it corresponds via edge labels to a path in the graph starting at the root.

The \define{follower set} of the word $w \in \lang$ is the set 
\[
\foll(w) = \sigma^{|w|}([w]) = \{ x \in \{0,\dots,\ell\}^\N : wx \in \Sigma \}
\]
of sequences which can be concatenated to $w$ to give a sequence in $\Sigma$.
It is known from work of Takahashi~\cite{MR0340552} and Hofbauer~\cite{MR0492180} that every subshift can be encoded graphically using follower sets as vertices.
For intermediate beta shifts this procedure can be simplified by characterizing follower sets using the coordinates
\[
\begin{aligned}
k_1(w) &= \max \{0\} \cup \{ 1 \le k \le |w| : (w_{|w| - k + 1},\dots,w_{|w|}) = (a_1,\dots,a_k) \} \\
k_2(w) &= \max \{0\} \cup \{ 1 \le k \le |w| : (w_{|w| - k + 1},\dots,w_{|w|}) = (b_1,\dots,b_k) \}
\end{aligned}
\]
which record the lengths of the longest tail segments of $w$ that agree with initial segments of $a$ and $b$ respectively.
One can show that
\begin{equation}
\label{eqn:alphaBetaFollower}
\foll(w) = \{ x \in \Sigma : \sigma^{k_1(w)}(a) \preceq x \preceq \sigma^{k_2(w)} (b) \}
\end{equation}
for all $w$ so follower sets are indexed by the points of $(\N \cup \{0\})^2$.

We are now ready to define the graph.
For each $j,k \in \N \cup \{0\}$ with
\[
\{ x \in \Sigma : \sigma^j (a) \preceq x \preceq \sigma^k (b) \} \ne \emptyset
\]
the pair $(j,k)$ defines a vertex of $\Gamma$ denoted by $[j,k]$.
There are four types of edge.
\begin{enumerate}
[label=\textbf{E\arabic*}.,ref=\textbf{E\arabic*}]
\item
\label{edge:both}
(Follow $a$ and $b$)
If $\vertex{j,k}$ is a vertex and $a_{j+1} = b_{k+1}$ then this quantity labels an edge from $\vertex{j,k}$ to $\vertex{j+1,k+1}$.
\item
\label{edge:a}
(Follow $a$ alone)
If $\vertex{j,k}$ is a vertex and $a_{j+1} < b_{k+1}$ then $a_{j+1}$ labels an edge from $\vertex{j,k}$ to $\vertex{j+1,0}$.
\item
\label{edge:b}
(Follow $b$ alone)
If $\vertex{j,k}$ is a vertex and $a_{j+1} < b_{k+1}$ then $b_{k+1}$ labels an edge from $\vertex{j,k}$ to $\vertex{0,k+1}$.
\item
\label{edge:reset}
(Reset)
If $\vertex{j,k}$ is a vertex every $a_{j+1} < c < b_{k+1}$ labels and edge from $\vertex{j,k}$ to $\vertex{0,0}$.
\end{enumerate}
There is always an edge labeled $0$ from $\vertex{0,0}$ to $\vertex{1,0}$ and an edge labeled $\ell$ from $\vertex{0,0}$ to $\vertex{0,1}$.
We also have, for every $0 < i < \ell$, an edge from $\vertex{0,0}$ to $\vertex{0,0}$ labeled $i$.
See Figure~\ref{fig:graphExample} for a partial example.

\begin{figure}[htb]
\centering
\begin{tikzcd}[column sep=small, row sep=scriptsize]
\\
&
&
{\svtx{2,10}}
\\
&
{\svtx{1,9}}
\arrow[swap]{ur}{0}
\\
{\svtx{0,8}}
\arrow[swap]{ur}{0}
\\
&
{\svtx{1,7}}
\arrow[swap]{ul}{1}
\arrow[loop, out=45, in=30, distance=3cm, near start]{dddddddr}{0}
\\
{\svtx{0,6}}
\arrow[swap]{ur}{0}
\\
{\svtx{0,5}}
\arrow[swap]{u}{1}
\arrow[loop, blue, out=0, in=30, distance=2cm]{dddddr}
\\
&
{\svtx{1,4}}
\arrow[swap]{ul}{2}
\arrow[swap,loop, out=180, in=150, distance=3cm,overlay, very near start]{ddddl}{1}
\arrow[loop, out=45, in=60, distance=2.5cm, near start]{ddddr}{0}
\\
{\svtx{0,3}}
\arrow[swap]{ur}{0}
\\
{\svtx{0,2}}
\arrow[swap]{u}{2}
\arrow[loop, out=225, in=135,swap, near start]{dd}{1}
\arrow[blue,loop, out=330, in=60, distance=0.5cm]{ddr}
&
&
&
&
&
&
&
&
&
{\svtx{9,2}}
\arrow[swap]{rdd}{0}
\arrow[loop, out=270, in=290,overlay, very near start]{ddlllllllll}{1}
\arrow[loop,out=150, in=0,very near start, swap]{ulllllllll}{2}
\\
{\svtx{0,1}}
\arrow[swap]{u}{1}
\arrow[blue,loop, out=0, in=90, distance=0.25cm]{dr}
&
&
&
&
&
{\svtx{5,1}}
\arrow[swap]{dr}{0}
\arrow[loop,out=150, in=0, very near start,swap]{lllllu}{1}
&
&
&
{\svtx{8,1}}
\arrow[swap]{ur}{1}
\\
{\svtx{0,0}}
\arrow[swap]{r}{0}
\arrow[swap]{u}{2}
\arrow[loop, out=180, in=270, distance=1.5cm,overlay,swap]{}{1}
&
{\svtx{1,0}}
\arrow[swap]{r}{0}
\arrow[red,loop,out=120,in=340]{ul}
\arrow[loop,out=225,in=315]{l}{1}
&
{\svtx{2,0}}
\arrow[swap]{r}{1}
\arrow[red,swap,loop,out=90,in=18,distance=0.2cm]{llu}
&
{\svtx{3,0}}
\arrow[swap]{r}{1}
\arrow[red,swap,loop,out=90,in=36,distance=0.3cm]{lllu}
&
{\svtx{4,0}}
\arrow[swap]{ru}{2}
&
&
{\svtx{6,0}}
\arrow[swap]{r}{1}
\arrow[red,swap,loop,out=90,in=54,distance=3cm]{llllllu}
&
{\svtx{7,0}}
\arrow[swap]{ur}{2}
&
&
&
{\svtx{10,0}}
\\
\\
\\
\\
\\
\\
\end{tikzcd}
\caption{Part of the graph $\Gamma$ when $a = 0011201210\cdots$ and $b = 2120210100\cdots$ with all edges terminating at $\vertex{0,1}$ labeled ``2'' and all edges terminating at $\vertex{1,0}$ labeled ``0''.}
\label{fig:graphExample}
\end{figure}
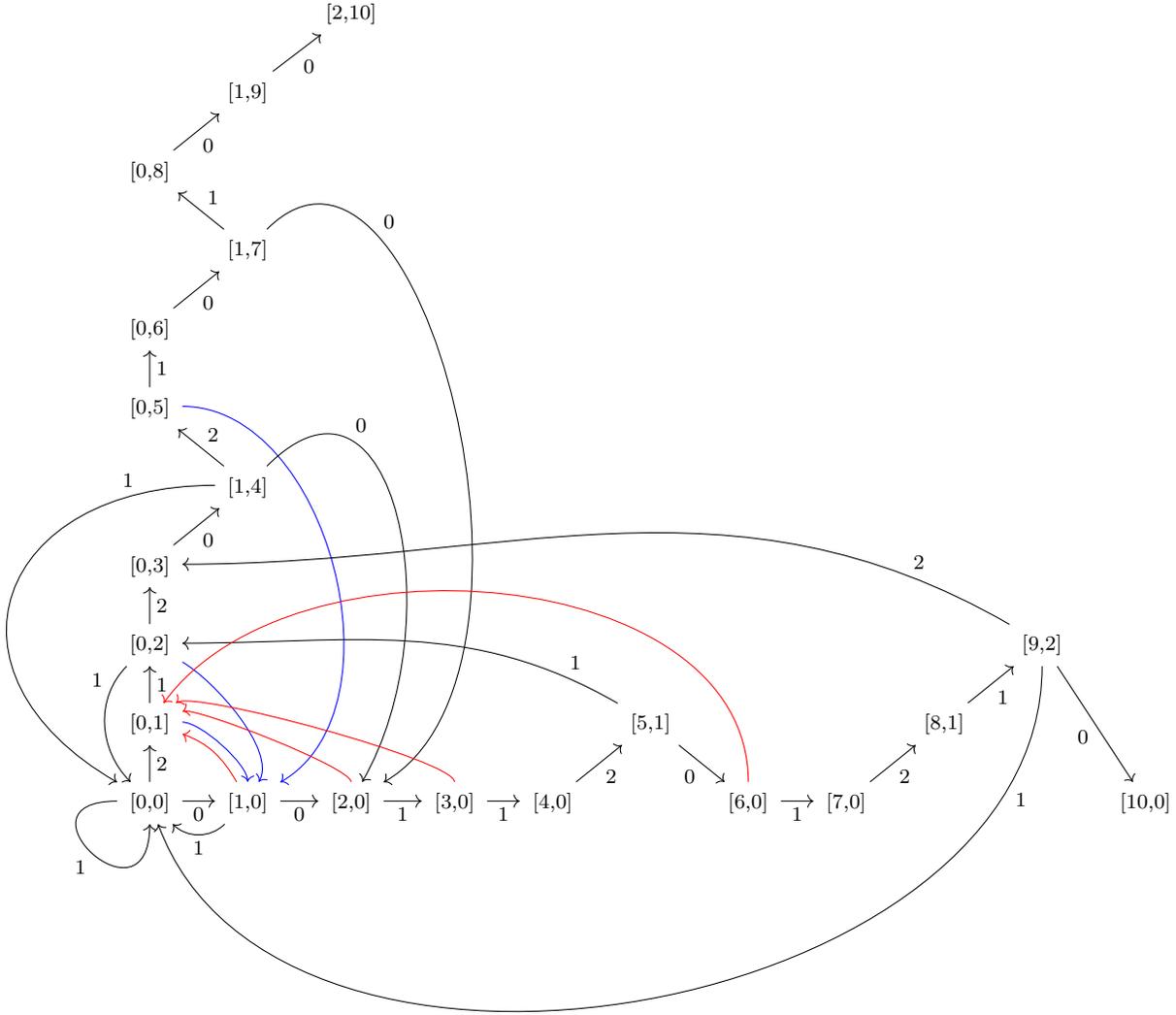

\begin{lemma}
Every vertex $\vertex{j,k}$ lies at the end of a path starting at $\vertex{0,0}$ and labeled by either an initial segment of $a$ or an initial segment of $b$.
\end{lemma}
\begin{proof}
If $\vertex{j,k}$ is a vertex then there is a word $w$ such that $(w_{|w|-j+1},\dots,w_{|w|}) = (a_1,\dots,a_j)$ and $(w_{|w| - k + 1},\dots,w_{|w|}) = (b_1,\dots,b_k)$.
If $j \ge k$ then $(a_1,\dots,a_j)$ labels the desired path, otherwise use $(b_1,\dots,b_k)$.
\end{proof}

\begin{lemma}
The only vertex of the form $\vertex{k,k}$ is $\vertex{0,0}$.
\end{lemma}
\begin{proof}
No word has a tail segment of length $k \in \N$ agreeing with $(a_1,\dots,a_k)$ and $(b_1,\dots,b_k)$ because $a_1 \ne b_1$.
\end{proof}

\begin{lemma}
\label{lem:uniqueLeaver}
For every $N \in \N$ there is only one vertex of the form $\vertex{N,j}$ with $j \le N$.
\end{lemma}
\begin{proof}
By the previous lemma, if $\vertex{N,j}$ is a vertex with $j \le N$ then $j < N$.
All words $w$ with $k_1(w) = N$ and $k_2(w) < N$ have the same $k_2$ coordinate.
\end{proof}

\begin{lemma}
\label{lem:graphpropone}
Let $\gamma$ be a path in $\Gamma$ starting at $\vertex{0,0}$ and ending at $\vertex{i,j}$.
If $w$ is the associated word over $\{0,\dots,\ell\}$ then $i$ is maximal such that $(w_{n-i+1},\dots,w_n) = (a_1,\dots,a_i)$ and $j$ is maximal such that $(w_{n-j+1},\dots,w_n) = (b_1,\dots,b_j)$.
\end{lemma}
\begin{proof}
We proceed by induction on the length $n$ of the path.
When $n = 1$ there are four cases to check.
If $w_1 = a_1$, then $w$ labels an edge from $\vertex{0,0}$ to $\vertex{1,0}$.
If $a_1 < w_1 < b_1$ then $\gamma$ is a loop at $\vertex{0,0}$.
If $w_1 = b_1$, then $w$ labels an edge from $\vertex{0,0}$ to $\vertex{0,1}$.

Now suppose the conclusion of the lemma holds for paths of length $n$.
Let $\gamma$ be a path in $\Gamma$ of length $n+1$ starting at $\vertex{0,0}$.
Let $\vertex{i,j}$ be the penultimate vertex of $\gamma$ and let $w$ be the corresponding word over $\{0,\dots,\ell\}$.
There are four cases to consider according to which of the rules \ref{edge:both} through \ref{edge:reset} determines the last edge of $\gamma$. 

In case \ref{edge:both} the path $\gamma$ ends at $\vertex{i+1,j+1}$.
By the induction hypothesis, $i$ is maximal such that $(w_{n-i+1},\dots,w_n) = (a_1,\dots,a_i)$.
Hence, $a_{i+1} = w_{n+1}$ implies $i+1$ is maximal such that $(w_{n-i+1},\dots,w_{n+1}) = (a_1,\dots,a_{i+1})$.
By a similar argument $j+1$ is maximal with $(w_{n-j+1},\dots,w_{n+1}) = (b_1,\dots,b_{j+1})$.

In case \ref{edge:a} our path $\gamma$ ends at $\vertex{i+1,0}$.
As in the previous case $i+1$ is maximal such that $(w_{n-i+1},\dots,w_{n+1}) = (a_1,\dots,a_{i+1})$.
What needs to be shown is that $(w_k,\dots,w_{n+1}) \ne (b_1,\dots,b_{n-k+2})$ for all $1\le k\le n+1$.
That is, the word $w$ does not end with an initial sub-word of $b$.
By the induction hypothesis, $j$ is maximal such that $(w_{n-j+1},\cdots,w_n) = (b_1,\dots,b_j)$.
This implies $(w_k,\dots,w_{n+1}) \ne (b_1,\dots,b_{n-k+2})$ for all $1\le k \le n-j$.
For $n-j+1\le k \le n+1$, we have $w_k^{n} = b_{k+j-n}^j$.  This leads to
\[
w_{k}^{n+1}
=
b_{k+j-n}^{j}w_{n+1}
\prec
b_{k+j-n}^{j+1}
\preceq
b_1^{n-k+2} 
\]
with the second inequality holding because $b_1^{j+1}$ satisfies \eqref{eqn:wordorder}.

The remaining cases \ref{edge:b} and \ref{edge:reset} are similar.
\end{proof}

\begin{lemma}
\label{lem:graphproptwo}
A word $w$ over $\{0,\dots,\ell\}$ is in $\lang$ if and only if the letters of $w$ label a path in $\Gamma$ that begins at $\vertex{0,0}$.
\end{lemma}
\begin{proof}
We proceed by induction on $n$.
The case $n=1$ is immediate: for every $i$ in $\{0,\dots,\ell\}$ there is an edge labeled $i$ leaving $\vertex{0,0}$.

Suppose the conclusion of the lemma holds for words of length $n$ over $\{0,\dots,\ell\}$.
Let $w$ be a word of length $n+1$ over $\{0,\dots,\ell\}$.
If the initial subword $(w_1,\dots,w_n)$ is not in $\lang$, then by the induction hypothesis $(w_1,\dots,w_n)$ does not spell a path in $\Gamma$.
Hence, neither is $w$ in $\lang$ nor does $w$ spell a path in $\Gamma$.
If $(w_1,\dots,w_n)$ is in $\lang$ the by induction there is a corresponding path $\gamma$ of length $n$ in $\Gamma$ starting at $\vertex{0,0}$.
Let $\vertex{i,j}$ be the terminal vertex of $\gamma$.
Certainly $a_{i+1} \le w_{n+1} \le b_{j+1}$.
We conclude by using Corollary \ref{cor:lexword} to show that if $w_1^n$ is in $\lang$, then $w_1^{n+1}$ is in $\lang$ if and only if $a_{i+1}\le w_{n+1}\le b_{j+1}$.

If $(w_1,\dots,w_{n+1})$ is in $\lang$, then Corollary~\ref{cor:lexword} implies $(a_1,\dots,a_{i+1}) \preceq (w_{n-i+1},\dots,w_{n+1})$.
The path corresponding to $(w_1,\dots,w_n)$ ends at the vertex $\vertex{i,j}$ so Lemma \ref{lem:graphpropone} implies $i$ is maximal such that $w_{n-i+1}^n = a_1^i$.
Evidently
\[
(a_1,\dots,a_{i+1}) \preceq (a_1,\dots,a_i,w_{n+1})
\]
hence it must be that $w_{n+1}\ge a_{i+1}$.  A similar argument shows that $w_{n+1}\le b_{j+1}$.

Suppose $w_{n+1} \ge a_{i+1}$.
By the induction hypothesis and Corollary~\ref{cor:lexword} we have
\[
(a_1,\dots,a_{k-n+1}) \preceq (w_k,\dots,w_n)
\]
for $1\le k \le n$.
By maximality of $i$, we have $(a_1,\dots,a_{k-n+1}) \prec (w_k,\dots,w_n)$.  Hence
\[
(a_1,\dots,a_{k-n+2}) \prec (w_k,\dots,w_{n+1})
\]
for all $1\le k \le n-i$.
If $n-i+1\le k \le n+1$ then
\[
(a_1,\dots,a_{n-k+2})
\preceq
(a_{i+k-n},\dots,a_{i+1})
\preceq
(a_{i+k-n},\dots,a_i,w_{n+1})
=
(w_k,\dots,w_{n+1})
\]
with the first inequality holding because $a_1^{i+1}$ satisfies \eqref{eqn:wordorder}.
A similar argument shows that $w_{n+1}\le b_{j+1}$ implies $(w_k,\dots,w_{n+1}) \preceq (b_1,\dots,b_{n-k+2})$ for all $1\le k \le n+1$.
Hence, $a_{i+1}\le w_{n+1}\le b_{j+1}$ which implies $w$ is in $\lang$. 
\end{proof}

To every path in $\Gamma$ we can associate a word over $\{0,\dots,\ell\}$ by reading off the labels of the edges in the path.
By Lemma~\ref{lem:graphproptwo}, every word $w$ in the language $\lang$ of $\Sigma$ corresponds to a path $\pth(w)$ in $\Gamma$ beginning at $\vertex{0,0}$.
Write $\vtx(w)$ for the terminal vertex of the path $\pth(w)$.
Thus $\vtx(w) = \vertex{k_1(w),k_2(w)}$.
We also write $\pth(x)$ for the infinite path in $\Gamma$ starting at $\vertex{0,0}$ determined by $x \in \Sigma$.

By a \define{flat} in the graph we mean a path either of the form
\[
\vertex{r,0} \to \vertex{r+1,0} \to \cdots \to \vertex{s,0}
\]
or of the form
\[
\vertex{0,r} \to \vertex{0,r+1} \to \cdots \to \vertex{0,s}
\]
for some $r < s$ in $\N \cup \{0\}$.
Paths of the first kind are called \define{vertical flats} while paths of the second kind are called \define{horizontal flats}.
Note that the requirement $r < s$ precludes a single vertex from being considered a flat.

\begin{lemma}
\label{lem:flatRestrictions}
No edge in a horizontal flat is labeled $0$ and no edge in a vertical flat is labeled $\ell$.
\end{lemma}
\begin{proof}
If an edge with source $\vertex{0,k}$ is labeled $0$ then its target is either $\vertex{1,k+1}$ or $\vertex{1,0}$.
Neither of these vertices can belong to a horizontal flat.
Similarly, if an edge with source vertex $\vertex{j,0}$ is labeled $\ell$ then its target is either $\vertex{j+1,1}$ or $\vertex{1,0}$.
\end{proof}

By a \define{diagonal} in the graph we mean a path of the form
\begin{equation}
\label{path:diagonal}
\vertex{p,q} \to \vertex{p+1,q+1} \to \cdots \to \vertex{p+d,q+d}
\end{equation}
for some $p,q,d \in \N$.

\begin{lemma}
If $\vertex{p,q} \to \vertex{p+1,q+1}$ is an edge then no other edge has $\vertex{p,q}$ as a source.
\end{lemma}
\begin{proof}
Since there is an edge from $\vertex{p,q}$ to $\vertex{p+1,q+1}$ we must have $a_{j+1} = b_{k+1}$.
Thus none of the possibilities \ref{edge:a}, \ref{edge:b} or \ref{edge:reset} occur at $\vertex{p,q}$.
\end{proof}

If \eqref{path:diagonal} is a maximal diagonal (i.e.\ not a sub-path of a longer diagonal) then there are at least two edges with source $\vertex{p+d,q+d}$.
These edges have targets $\vertex{p+d+1,0}$ and $\vertex{0,q+d+1}$.
If we have $a_{p+d} < b_{q+d}$ there will also be edges from $\vertex{p+d,q+d}$ to $\vertex{0,0}$.

By a \define{vertical reset} in the graph we mean any edge of the form $\vertex{p,q} \to \vertex{0,q+1}$ and by a \define{horizontal reset} we mean any edge of the form $\vertex{p,q} \to \vertex{p+1,0}$.
The infinite path in $\Gamma$ that corresponds to $b$ is made up of horizontal flats, diagonals and vertical resets.
It may be that some of these flats are empty (in otherwords a diagonal may immediately follow a reset) and it may be that $b$ is coterminal with an infinite horizontal flat or diagonal.

\subsection{Characterizing specification}

Fix $\beta > 2$.
The purpose of this section is to characterize in terms of $a$ and $b$ when the subshift $\Sigma$ has specification to prove i.e.\ to prove Theorem~\ref{thm:specChar}.
We begin with some preliminary results.

\begin{proposition}
\label{prop:rts}
If $\beta > 2$ then for every interval $I \subset [0,1)$ there is $\tau \in \N$ and a subinterval $L \subset I$ such that $F^{\tau}(L) = [0,1)$ and $F^{\tau}$ restricted to $L$ is continuous.
\end{proposition}
\begin{proof}
Without loss of generality, we may assume $I$ is contained in a single interval of continuity of $F$.
Let $|I|$ denote the length of $I$.
We define two sequences $I_0, I_1, I_2,\dots$ and $L_0 \supset L_1\supset L_2\supset \cdots$ of intervals by the following recursive process.

Put $I_0 = I$ and $L_0 = I$.
Define $I_1$ to be a subinterval of $F(I_0)$ of maximal length that is contained in an interval of continuity of $F$.
The interval $L_1 \subset L_0$ then consists of those points in $L_0$ that map to $I_1$ under $F$.
Inductively, if $I_n$ and $L_n$ have been defined, let $I_{n+1}$ be an subinterval of $F(I_n)$ of maximal length that is contained in an interval of continuity of $F$ and let  $L_{n+1} \subset L_n$ be those points that map to $I_{n+1}$ under $F^{n+1}$.

By design $F^{n+1}$ is continuous when restricted to $L_n$ for all $n$.
If, for some $n$, the interval $F(I_n)$ contains two or more points of discontinuity of $F$ then $F^{n+2}(I_{n+1}) = [0,1)$ and we get the result with $L = L_{n+1}$ and $\tau = n+2$.
If $F(I_n)$ contains fewer than two points of discontinuity of $F$, then $|I_{n+1}| \ge \frac{\beta}{2}|I_n|$.
It must then be the case that $F(I_n)$ contains two or more points of discontinuity of $F$ for some $n$, otherwise $\beta > 2$ implies the size of the intervals $I_n$ would grow without bound.
\end{proof}

\begin{corollary}
\label{cor:symbrts}
If $\beta > 2$ then for every word $w \in \lang$ there exists a word $u \in \lang$ such that $w u x \in \Sigma$ for every $x \in \Sigma$.
\end{corollary}
\begin{proof}
Fix $w,v \in \lang$.
The cylinder sets $[w]$ and $[v]$ correspond to intervals $I_w$ and $I_v$ in $[0,1)$ of points whose encodings begin with $w$ and $v$, respectively.
Apply Proposition~\ref{prop:rts} to the interval $I_w$ to get $\tau \in \N$ and $L \subset I_w$ with the stated properties.
Since the intervals of continuity of $F$ have length at most $\frac{1}{\beta}$, we must have $|I_w| \le \frac{1}{\beta^{|w|}}$.
Hence $\tau \ge |w|$.

From $F^{\tau}(L) = [0,1)$ we get, for every $v \in \lang$, some $u\in \lang_{\tau-|w|}$ with $w u v \in \lang$.
Continuity of $F^{\tau}$ restricted to $L$ implies $u$ is independent of $v$.
Given $x \in \Sigma$ we can choose $v$ to be an arbitrarily long initial subword of $x$.
This implies $w u x$ is in $\Sigma$ because $\Sigma$ is closed.
\end{proof}

\begin{corollary}
\label{cor:graphrts}
If $\beta > 2$ then for every vertex $\xi$ in $\Gamma$ there is a path from $\xi$ to $\vertex{0,0}$. 
\end{corollary}
\begin{proof}
Given a vertex $\xi \in \Gamma$ let $w \in \lang$ be any word with $\xi = \vertex{k_1(w),k_2(w)}$.
By Corollary~\ref{cor:symbrts} there is a word $u \in \lang$ such that the path $\pth(wu)$ in $\Gamma$ ends at a vertex that has an outgoing edge for every member of $\{0,\dots,\ell\}$.
Now $\ell \ge 2$ so \ref{edge:reset} implies that whenever a vertex has at least three outgoing edges, one of the edges terminates at $\vertex{0,0}$.
\end{proof}

We are now ready for the proof of Theorem~\ref{thm:specChar}.

\begin{proof}[Proof of Theorem~\ref{thm:specChar}]
Let $L = \max \diag(a)$ and let $N = \max \diag(b)$.
By Lemma~\ref{lem:graphpropone} every vertex $\vertex{i,j}$ of $\Gamma$ satisfies $i < N+1$ or $j < L+1$.
We claim that $\vtx(a_1^{N+1})$ is the only vertex with first coordinate $N+1$ and $\vtx(b_1^{L+1})$ is the only vertex with second coordinate $L+1$.
We prove the statement for $\vtx(a_1^{N+1})$; the proof for $\vtx(b_1^{L+1})$ is similar.
Lemma~\ref{lem:graphpropone} implies $\vtx(a_1^{N+1}) = \vertex{N+1,j}$ where $j$ is maximal such that $a_{N-j+2}^{N+1}=b_1^j$.
Lemma~\ref{lem:graphpropone} also implies that a word $w \in\lang_n$ with $\vtx(w) = \vertex{N+1,j'}$ satisfies $w_{n-N}^n = a_1^{N+1}$ and $j'$ is maximal such that $w_{n-j'+1}^n = b_1^{j'}$.
Since $b$ does not contain $a_1^{N+1}$ it must be that $j'<N-1$, but then $j'$ is maximal such that $b_1^{j'}$ is an ending subword of $a_1^{N+1}$ so $j'=j$.

From \ref{edge:a} we see that from a vertex $\vertex{i,j}$ with $i<N+1$ the word $a_{i+1}^{N+1}$ spells a path to a vertex with first coordinate $N+1$.
By the claim in the previous paragraph, this vertex is $\vtx(a_1^{N+1})$.
Similarly, from a vertex $\vertex{i,j}$ with $j<L+1$ the word $b_{j+1}^{L+1}$ spells a path to $\vtx(b_1^{L+1})$.
By Corollary~\ref{cor:graphrts}, there are paths $\eta_a$ and $\eta_b$ from $\vtx(a_1^{N+1})$ and $\vtx(b_1^{L+1})$ to $\vertex{0,0}$ with lengths we denote by $\tau_a$ and $\tau_b$.
These observations imply that from every vertex in $\Gamma$ there is a path to $\vertex{0,0}$ of length at most $\tau = \max\{N+1,L+1\} + \max\{\tau_a,\tau_b\}$.
By Lemma \ref{lem:graphproptwo}, the language of $\Sigma$ has specification with gap time $\tau$.

On the other hand, suppose $a$ contains arbitrarily long initial subwords of $b$.
(When $b$ contains arbitrarily long initial subwords of $a$ the argument is similar.)
We fix $\tau \in \N$ and show that there are words $w,v \in\lang$ for which $wuv \notin \lang$ for all $u\in\lang_{\tau}$.
By assumption, we can find $b_1^K$ in $a$ with $K > \tau$.  Assume the initial subword $b_1^K$ of $b$ begins in $a$ at index $q+1$.
That is, assume $a = a_1^{q} b_1^K a_{q+K+1}\cdots$.
By Lemma~\ref{lem:graphpropone}, if $q$ is chosen as small as possible, then $\vtx(a_1^{q+K}) = \vertex{q+K,K}$.
It follows that $\vtx(a_1^q) = \vertex{q,0}$ and that $b_1^K$ spells a path from $\vertex{q,0}$ to $\vertex{q+K,K}$.
Since both coordinates are increasing along this path, it follows that none of the vertices seen have any other outgoing edges.
By Lemma \ref{lem:graphproptwo}, if we let $w=a_1^q$, then $u=b_1^\tau$ is the only word in $\lang_{\tau}$ such that $wu\in\lang$ and $b_{\tau+1}$ is the only letter in $\{0,\dots,\ell\}$.
Thus we get the result by letting $v$ be any word in $\lang$ whose first letter is not $b_{\tau+1}$. 
\end{proof}

\section{Proof of Theorem~\ref{thm:mainTheorem}}
\label{sec:proof}

In this section we prove Theorem~\ref{thm:mainTheorem}.
We therefore fix $0 \le \alpha < 1$ and $\beta > 2$ and a \Holder{} continuous function $\phi : \Sigma \to \R$.
By hypothesis there is a bound $L \in \N$ on the set $\diag(a)$.
Our main tool in the proof of Theorem~\ref{thm:mainTheorem} is the following result giving a sufficient condition for a subshift to have unique equilibrium states.

\begin{theorem}
[{\cite[Theorem~C]{MR3046278}}]
\label{thm:use}
Let $\Sigma$ be a subshift on a finite alphabet with language $\lang$.
Suppose $\pref,\good,\suff \subset \lang$ satisfy $\pref \good \suff = \lang$.
Put
\[
\good(M) = \{ vwu \in \lang : v \in \pref, w \in \good, u \in \suff, |v| \le M , |u| \le M \}
\]
for each $M \in \N$.
A continuous function $\phi : \Sigma \to \R$ has a unique equilibrium state if all of the following conditions hold.
\begin{enumerate}
[label=\textup{\textbf{CT\arabic*}}.,ref=\textup{\textbf{CT\arabic*}}]
\item
\label{ct:spec}
The set $\good(M)$ has specification for all $M \in \N$.
\item
\label{ct:bowen}
The potential $\phi$ has the Bowen property.
\item
\label{ct:pressGap}
$\press(\phi,\pref \cup \suff) < \press(\phi,\lang)$.
\end{enumerate}
\end{theorem}

\subsection{A Language decomposition}
\label{sec:lang}

Let $\good \subset \lang$ consist of all words $w$ such that $\vtx(w)$ has second coordinate equal to $0$.
Let $\suff \subset \lang$ consist of the initial subwords of $b$.
Thus $\suff_n = \{ b_1^n \}$.
Allowing for the empty word, we have the decomposition $\lang = \good \suff$.
For $M \in \N$ we denote by $\good(M) \subset \lang$ the set of all words $w$ that can be written $w = gs$ with $g\in\good$ and $s \in \suff$ with $|s| \le M$.
Equivalently, $w \in \good(M)$ if and only if the second coordinate of $\vtx(w)$ is at most $M$.  

\begin{proposition}
\label{prop:specset}
If $\beta > 2$ and $\diag(a)$ is bounded then $\good(M)$ has specification for all $M \in \N$.
\end{proposition}
\begin{proof}
Let $b_1^L$ be the longest initial subword of $n$ that appears in $a$.
It suffices to consider the case $M>L$.
By Lemma~\ref{lem:graphproptwo} it suffices to show that there exists a $\tau \in \N$ such that from any vertex with second coordinate at most $M$, there is a path to $\vertex{0,0}$ of length at most $\tau$.  

It follows from Lemma \ref{lem:graphpropone} that for all $K>L$ the only vertex with second coordinate equal to $K$ is $\vtx(b_1^K)$.
Indeed, if $\vtx(w) = \vertex{i,K}$ for some $w\in\lang$, then $i$ and $K$ are maximal such that $a_1^i$ and $b_1^K$ are ending subwords of $w$.
Since $a$ does not contain $b_1^K$, it must be that $i < K$.  Hence, $a_1^i$ is an ending subword of $b_1^K$ and is maximal in this regard.
This implies the first coordinate of $\vtx(b_1^K)$ is also $i$, hence $\vtx(w) = \vtx(b_1^K)$.
In particular, $\vtx(b_1^{M})$ is the only vertex with second coordinate $M$.

By Corollary~\ref{cor:graphrts}, there is a path $\eta$ from  $\vtx(b_1^{M})$ to $\vertex{0,0}$.
Let the length of this path be $\tau_M$.
For every $j < M$ the word $b_{j+1}^M$ spells a path from any vertex $\vertex{i,j}$ to a vertex with second coordinate $M$, which must be $\vtx(b_1^{M})$.
This shows that from any vertex with second coordinate at most $M$, there is a path to $\vertex{0,0}$ of length at most $\tau_M+M$.
\end{proof}

\begin{corollary}
\label{cor:specset}
If $x \in \Sigma$ does not contain arbitrarily long initial subwords of $b$, then $x$ belongs to a subshift of $\Sigma$ that has specification.  
\end{corollary}
\begin{proof}
Let $b_1^K$ be the longest initial subword of $b$ that appears in $x$ and choose $M > \max\{ K, L \}$.
Let $\Sigma' \subset \Sigma$ contain every sequence with the property that the corresponding path in $\Sigma$ never visits vertices with second coordinate greater than $M$.
By Lemma~\ref{lem:graphpropone}, the set $\Sigma'$ is shift invariant and contains $x$.
If a sequence $y$ is not in $\Sigma'$, then again by Lemma~\ref{lem:graphpropone}, we have
\[
(y_n,\dots,y_{n+M+1}) = (b_1,\dots,b_{M+1})
\]
for some $n \in \N$.
The open cylinder
\[
\{z \in \Sigma : (z_n,\dots,z_{n+M+1}) = (b_1,\dots,b_{M+1}) \}
\]
contains $y$ and is disjoint from $\Sigma'$.
Hence, $\Sigma'$ is closed.
The language of $\Sigma'$ has specification because the language of $\Sigma'$ is contained in $\mathcal{G}(M)$ and this has specification by Proposition~\ref{prop:specset}. 
\end{proof}

\subsection{Comparison sequences}
\label{subsec:compseq}

The sequence $b$ does not end in an infinitely long diagonal by Corollary~\ref{cor:graphrts}.
Therefore the path corresponding to $b$ has either infinitely many resets or ends in an infinite flat.
In other words $b$ has the form
\begin{equation}
\label{eqn:bseq}
b
=
\underbrace{b_1 \cdots b_{n_1}}_{\textup{flat}\vphantom{dgfl}}
\underbrace{a_1 \cdots a_{m_1}}_{\textup{diagonal}\vphantom{dgfl}}
\underbrace{b_{n_1 + m_1 + 1}}_{\textup{reset}\vphantom{dgfl}}
\underbrace{b_{n_1 + m_1 + 2} \cdots b_{n_2}}_{\textup{flat}\vphantom{dgfl}}
\underbrace{a_1 \cdots a_{m_2}}_{\textup{diagonal}\vphantom{dgfl}}
\underbrace{b_{n_2 + m_2 + 1}}_{\textup{reset}\vphantom{dgfl}}
\cdots
\end{equation}
for some $m_i,n_i \in \N$ with $n_{i+1} \ge n_i + m_i + 1$ for all $i$.
Note that we may have $n_{i+2} = n_i + m_i + 1$ in which case $b$ transitions immediately from a reset to a diagonal, and that $b$ may have only finitely many diagonals.
The sequence $b$ therefore satisfies (exactly) one of the following criterion.
\begin{enumerate}
[label=\textup{\textbf{B\arabic*}}.,ref=\textup{\textbf{B\arabic*}}]
\item
\label{b:flats}
There are infinitely many $n$ with $b_n$ belonging to a flat of $b$ and $b_n \ne 1$.
\item
\label{b:manyDiags}
All but finitely many of the $b_n$ belonging to flats equal $1$.
\end{enumerate}

We next define auxiliary sequences $c$ and $d$ that will be used in proving Theorem~\ref{thm:mainTheorem} according to whether we are in case \ref{b:flats} or \ref{b:manyDiags} respectively.
First, define $c$ by editing $b$ as follows: change every letter of $b$ belonging to a flat to a $1$.
Formally, if $Q$ is the set of indices $n$ such that $b_n$ corresponds to a flat in $\Gamma$, define
\[
c_n = \begin{cases} 1 & n \in E \\ b_n & \textup{otherwise} \end{cases}
\]
for all $n \in \N$.

\begin{lemma}
\label{lem:cAllowed}
The sequence $c$ belongs to $\Sigma$.
\end{lemma}
\begin{proof}
We prove that $c_1 \cdots c_n$ belongs to $\lang$ for all $n \in \N$.
Fix $n \in \N$ and let $m \le n$ be maximal with $m \in Q$ and $b_m \ne 1$.
(If no such $m$ exists then $c_1 \cdots c_n = b_1 \cdots b_n$ which certainly belongs to $\lang$.)
By \ref{edge:reset} there is an edge from $\vertex{0,m}$ to $\vertex{0,0}$ labeled 1.
Since every word corresponds to some path starting at $\vertex{0,0}$ the concatenation $b_1 \cdots b_{m-1} 1 b_{m+1} \cdots b_n$ belongs to $\lang$.
In particular $b_j \cdots b_{m-1} 1 b_{m+1} \cdots b_n$ belongs to $\lang$ for all $j < m-1$.
Repeating this argument for all indices $k$ at which $b_k \ne 1$ corresponds to an edge in a flat gives $c_1 \cdots c_n$ in $\lang$.
\end{proof}

The sequence $c$ has the form
\begin{equation}
\label{eqn:cseq}
c
=
\underbrace{1 \cdots 1\vphantom{b_{m_1}}}_{\textup{``flat''}\vphantom{dgfl}}
\underbrace{a_1 \cdots a_{m_1}}_{\textup{``diagonal''}\vphantom{dgfl}}
\underbrace{b_{n_1 + m_1 + 1}}_{\textup{``reset''}\vphantom{dgfl}}
\underbrace{1 \cdots 1\vphantom{b_{m_1}}}_{\textup{``flat''}\vphantom{dgfl}}
\underbrace{a_1 \cdots a_{m_2}}_{\textup{``diagonal''}\vphantom{dgfl}}
\underbrace{b_{n_2 + m_2 + 1}}_{\textup{``reset''}\vphantom{dgfl}}
\cdots
\end{equation}
where our labels of flat, diagonal and reset are now only heuristic, as $\pth(c)$ may not truly traverse the flats and diagonals of $\pth(b)$.

Our second sequence $d$ will be used in case \ref{b:manyDiags} and is defined as follows.
Fix by Corollary~\ref{cor:graphrts} a path $\gamma$ in $\Gamma$ from $\vertex{k_1(b_1^{L+1}),L+1}$ to $\vertex{0,0}$ and let $\eta$ be the corresponding word.
Define a sequence $d$ in $\{0,\dots,\ell\}^\N$ by simultaneously replacing every reset $b_{n_i + m_i + 1}$ in $b$ with the word $a_{m+i + 1} b_1 \cdots b_{L+1} \eta$.
Write $B_{n_i + m_i + 1} = a_{m+i + 1} b_1 \cdots b_{L+1} \eta$ for brevity.
The sequence $d$ has the form
\begin{equation}
\label{eqn:dseq}
d
=
\underbrace{b_1 \cdots b_{n_1}}_{\textup{``flat''}\vphantom{dgfl}}
\underbrace{a_1 \cdots a_{m_1}}_{\textup{``diagonal''}\vphantom{dgfl}}
\underbrace{B_{n_1 + m_1 + 1}}_{\textup{to } \vertex{0,0}\vphantom{dgfl}}
\underbrace{b_{n_1 + m_1 + 2} \cdots b_{n_2}}_{\textup{``flat''}\vphantom{dgfl}}
\underbrace{a_1 \cdots a_{m_2}}_{\textup{``diagonal''}\vphantom{dgfl}}
\underbrace{B_{n_2 + m_2 + 1}}_{\textup{to } \vertex{0,0}\vphantom{dgfl}}
\cdots
\end{equation}
where again our labels of flats and diagonals are now only heuristic, as $\pth(d)$ may not truly traverse the flats and diagonals of $\pth(b)$.

\begin{lemma}
\label{lem:dAllowed}
The sequence $d$ belongs to $\Sigma$.
\end{lemma}
\begin{proof}
The proof is similar to the proof of Lemma~\ref{lem:cAllowed}.
In particlar, it suffices to prove that if $x$ is obtained from $c$ by replacing one $b_{n_i+m_i+1}$ with the corresponding block $B_{n_i+m_i+1}$, then $x$ is allowed and the initial subword of $x$ through $B_{n_i+m_i+1}$ gives a path in $\Gamma$ that ends at $\vertex{0,0}$.
It suffices to show that the path in $\Gamma$ corresponding to the initial subword of $c$ through $a_1^{m_i}$ ends at a vertex $\vertex{m_i,j}$ for some $j$ and that $a_{m_i+1}$ labels an edge  from this vertex to $\vertex{m_i+1,0}$.
We proceed by induction on $i$.
For $i=1$, the initial subword $c_1\cdots a_1^{m_1}$ is some number of $1$'s followed by $a_1^{m_1}$ so the path corresponding to this subword ends at the same vertex the path corresponding to $a_1^{m_1}$ does, which is $\vertex{m_1,j}$ for some $j$.

By Lemmas \ref{lem:cAllowed} and \ref{lem:graphproptwo}, there is an outgoing edge at $\vertex{m_i,j}$ labeled $b_{n_1+m_1+1}$.
Since $b_{n_1+m_1+1}$ is lexicographically greater than $a_{m_1+1}$, we see from \ref{edge:a}, \ref{edge:b} that there are at least two outgoing edges at $\vertex{m_1,j}$, one labeled $a_{m_i+1}$ going to $\vertex{m_i+1,0}$ and one labeled $b_{n_1+m_1+1}$ to a vertex with first coordinate equal to $0$.  

If for some $i$ the path corresponding to the initial subword of $c$ through $b_{n_i+m_i+1}$ ends at a vertex with first coordinate equal to $0$, then it follows that the path corresponding to the initial subword of $c$ through $a_1^{m_{i+1}}$ ends at a vertex with first coordinate $m_{i+1}$.
There is always an outgoing edge labeled $a_{m_{i+1}+1}$ from such a vertex.
By Lemmas~\ref{lem:cAllowed} and \ref{lem:graphproptwo}, there is also an outgoing edge labeled $b_{n_{i+1}+m_{i+1}+1} > a_{m_{i+1}}$.
It must be that the former edge points to $\vertex{m_{i+1}+1,0}$ and the latter edge points to a vertex with first coordinate equal to $0$.
\end{proof}

\begin{lemma}
\label{lem:dspec}
The sequence $d$ is contained in a subshift of $\Sigma$ that has specification.
\end{lemma}
\begin{proof}
By Corollary \ref{cor:specset} it suffices to show that $d$ does not contain arbitrarily long initial segments of $b$.
This is equivalent to $\pth(d)$ only visiting vertices in $\Gamma$ with bounded second coordinate.
That this is the case is clear from the construction of $d$.
Indeed, the largest second coordinate of a vertex visited by $\pth(d)$ is the largest second coordinate of a vertex visited by the path $\eta$ in the definition of the $B_{n_i+m_i+1}$.
\end{proof}

\begin{lemma}
\label{lem:cspec}
Either $c$ is contained in a subshift of $\Sigma$ that has specification, or $b_i=c_i$ for all $i>L+1$. 
\end{lemma}
\begin{proof}
If $c$ is not contained in a subshift with specification then by Corollary \ref{cor:specset} it contains arbitrarily long initial segments of $b$.
For each $K > L$ fix an occurrence $(c_{N(K)+1},\dots,c_{N(K)+K})$ of the initial segment $b_1^K$ in $c$.
Since $b_1 = \ell$ we must have $c_{N(K)+1} \ne 1$ and therefore the edge corresponding to $b_{N(K)+1}$ belongs to either a diagonal of $b$ or a reset of $b$.
If that edge belongs to a diagonal of $b$ the it must be within the last $L$ edges of that diagonal because $K > L$ and we cannot have an initial segment of $b$ of length strictly greater than $L$ appearing in $a$.

By discarding some $K$ if necessary we may assume that $b_{N(K) + 1}$ is either in a reset of $b$ for all $K$ or the same distance $V$ from the end of a diagonal of $b$ for all $K$.
By then taking $K$ arbitrarily large we conclude that
\[
b
=
b_1 \cdots b_{V+1} 1^{k_1} a_1 \cdots a_{j_1} b_{l_1} 1^{k_2} a_1 \cdots a_{j_2} b_{l_2}\cdots
\]
holds.
It is possible that $k_1=0$ or that $b$ ends in infinitely many $1$'s.

A priori it may be that the segments $a_1 \cdots a_{j_i}b_{l_i}$ in the above description of $b$ overlap with the flats of $\pth(b)$.
To conclude that they do not, and therefore that all flats of $b$ after index $V+1$ are labeled by 1, it suffices to show that these segments are disjoint from the flats of $b$.
So suppose that $a_1 \cdots a_{j_i}$ appears somewhere in $b$.
Since $a_1 = 0$ and both flats and resets in $b$ are never labeled by zero, it must begin to appear in a maximal diagonal $b_{s+1} \cdots b_{s + m} = a_1 \cdots a_m$ of $b$, say at some position $b_{s+q}$ with $1 \le q \le m$.
We wish to verify that $j_i \le m-q+1$ for then our segment $a_1 \cdots a_{j_i} b_{l_i}$ will its end occurrence in $b$ before the next flat of $b$ begins.

Consider the letter $b_{s+m+1}$ that immediately follows the diagonal in which our segment begins to occur.
If $j_i > m-q+1$ then
\[
a_{m+1} < b_{s+m+1} = a_{m-q+2}
\]
whence $(a_q,\dots,a_{m+1}) \preceq (a_1,\dots,a_{m-q+1})$ contradicting the lexicographical condition \eqref{eqn:sequenceorder} that defines the language of $\Sigma$.
\end{proof}

We are now ready for the proof of Theorem~\ref{thm:mainTheorem}.
If $\good \subset \lang$ is as described in Section \ref{sec:lang}, then \ref{ct:spec} in Theorem \ref{thm:use} is satisfied by Proposition \ref{prop:specset}.
Certainly \ref{ct:bowen} is satisfied.
It remains to verify \ref{ct:pressGap}.
Let $V$ be the Bowen constant for $\phi$ and let $W=\max\{V,\nbar \phi \nbar_\infty \}$.
By Lemma~\ref{lem:sup} the inequality
\begin{equation}
\label{eqn:lastGoal}
\limsup_{n \to \infty} \frac{1}{n} (S_n \phi)(b) < \press(\phi,\lang)
\end{equation}
implies \ref{ct:pressGap}.
The proof of \eqref{eqn:lastGoal} depends on whether we are in case \ref{b:flats} or \ref{b:manyDiags}.

\subsection{Proof of Theorem~\ref{thm:mainTheorem} in case \ref{b:flats}}

Fix $\epsilon>0$ so that 
\begin{equation}
\label{eqn:pressureboundc}
\frac{1}{n} (S_n \phi)(c) < \press(\phi,\lang) - 5 \epsilon W 
\end{equation}
for sufficiently large $n$; this is possible by Lemma \ref{lem:sup} and Corollary \ref{cor:positiveentropy}.  

For each $n \in \N$ let $e_n$ be the number of letters in $b_1^n$ that are strictly larger than $1$ and correspond to edges in flats of $\pth(b_1^n)$.
By Lemma \ref{lem:cAllowed}, we can independently change these letters to $1$s and still have a word in $\lang$.
Let $\mathcal{W}$ be the set of all words that can be obtained by such edits.
We remark that $|\mathcal{W}_n| = 2^{e_n}$.

Fix $n \in \N$ and $w \in \mathcal{W}_n$.
Let $x \in \Sigma$ be any sequence with $x_1^n = w_1^n$.
When comparing $(S_n\phi)(b)$ and $(S_n\phi)(x)$ we can apply the Bowen property to the block of terms before the first edit, the block of terms after the last edit, and to the blocks of terms between consecutive edits.
Doing so, we see that the estimate
\begin{equation}
\label{eqn:bowenCedits}
|(S_n\phi)(b) - (S_n\phi)(x)|
\le
(k+1)V + 2k \nbar \phi \nbar_\mathsf{u}
\le
(3k+1)W
\end{equation}
holds where $k$ is the number of indices where $w_1^n$ and $b_1^n$ differ.

Now, we consider separately whether $e_n \ge \epsilon n$ or $e_n < \epsilon n$.
Suppose first that $e_n\ge\epsilon n$.
Then
\begin{align*}
\Lambda_n(\phi,\lang)
\ge
\Lambda_n(\phi,\mathcal{W})
&
\ge
\sum_{k=0}^{e_n} \binom{e_n}{k} e^{(S_n\phi)(b) - (3k+1)W}
\\
&
=
e^{(S_n\phi)(b)} e^{-W} \sum_{k=0}^{e_n} \binom{e_n}{k}e^{-3k}
\\
&
=
e^{(S_n\phi)(b)} e^{-W} (1+e^{-3W})^{e_n}
\end{align*}
by applying \eqref{eqn:bowenCedits} to each summand.
We obtain
\begin{equation}
\label{eqn:cequationone}
\begin{aligned}
\frac{1}{n}\log \Lambda_n(\phi,\lang)
&
\ge
\frac{1}{n} (S_n\phi)(b) - \frac{W}{n} + \frac{e_n}{n}\log(1+e^{-3W})
\\
&
\ge
\frac{1}{n} (S_n\phi)(b) - \frac{W}{n} + \epsilon\log(1+e^{-3W}) 
\end{aligned}
\end{equation}
after taking logs and dividing by $n$.

Suppose now that $e_n < \epsilon n$.
In this case we compare $(S_n \phi)(b)$ and $(S_n\phi)(c)$.
The words $b_1^n$ and $c_1^n$ differ in exactly $e_n$ places.
We obtain
\[
|(S_n\phi)(b) - (S_n\phi)(c)|
\le
(e_n+1)V + 2 e_n\nbar \phi \nbar_\mathsf{u}
\le
4e_nW
\]
by applying the Bowen property to the blocks between edits.
Consequently
\begin{equation}
\label{eqn:cequationtwo}
\frac{1}{n} (S_n\phi)(b)
\le
\frac{1}{n} (S_n\phi)(c) + \frac{4e_nW}{n}
\le
\frac{1}{n} (S_n\phi)(c) + 4 \epsilon W
\le
\press(\phi,\lang) - \epsilon W
\end{equation}
with the last inequality is an application of \eqref{eqn:pressureboundc}.

For all sufficiently large $n \in \N$ either \eqref{eqn:cequationone} or \eqref{eqn:cequationtwo} holds.
Therefore
\[
\limsup_{n \to \infty} \frac{1}{n} (S_n\phi)(b)
\le
\press(\phi,\lang) - \epsilon \min \{W, \log(1+e^{-3W}) \}
\]
giving \eqref{eqn:lastGoal}.

\subsection{Proof of Theorem~\ref{thm:mainTheorem} in case \ref{b:manyDiags}}

We will make use of the following estimate, which is a consequence of Stirling's formula.

\begin{lemma}
[{\cite[Lemma 5.4]{MR3624405}}]
\label{lem:logbinomial}
If $\delta n\le k\le \frac{n}{2}$ then $\log\binom{n}{k}\ge n\delta\log\frac{1}{\delta} - 2\log n$.
\end{lemma}

Let $e_n$ be the number of resets in $b_1^n$.
Write $N = |B_{n_i+m_i+1}| - 1 =  L + 1 + |\eta|$. 
When we change every reset $b_{n_i + m_i + 1}$ in $b_1^n$ to $B_{n_i + m_i + 1}$ we get a word of length $n+N e_n$.

We want to compare $(S_n\phi)(d)$ and $(S_n \phi)(b)$.
We begin by comparing $(S_{n+Ne_n}\phi)(d)$ and $(S_n \phi)(b)$.
By Lemma~\ref{lem:cspec}, for all $i$ such that $n_i+m_i+1>L+1$ we have that the block of $b_1^n$ between resets  $b_{n_i+m_i+1}$ and $b_{m_{i+1}+n_{i+1}+1}$ is identical to the block of $d_1^{n+D_nN}$ between $B_{n_i+m_i+1}$ and $B_{m_{i+1}+n_{i+1}+1}$.  There are \(D_n+1\) such blocks and we can apply the Bowen property to these.  All other terms can be bounded above in terms of $\nbar \phi \nbar_\mathsf{u}$, giving the estimate
\begin{align*}
|(S_n\phi)(b) - (S_{n+Ne_n}\phi)(d)|
&
\le
(e_n+1)V + 2(L+1)\nbar \phi \nbar_\mathsf{u} + 2e_n \nbar \phi \nbar_\mathsf{u} + e_nN \nbar \phi \nbar_\mathsf{u}
\\
&
\le
((3+N)e_n + 2(L+1)+1)W
\\
&
\le
3Ne_n W
\end{align*}
where, in the first estimate, the first term comes from the Bowen property, the second from not knowing what $b_1^{L+1}$ looks like, the third from the $e_n$ edits changing $b_{n_i+m_i+1}$ to $a_{m_i+1}$ and the last from the $e_n$ insertions of blocks of length $N$. 
From this we get
\begin{equation}
\label{eqn:bdCompare}
\begin{aligned}
|(S_n\phi)(b) - (S_n \phi)(d)|
&
\le
|(S_n \phi)(b) - (S_{n+Ne_n} \phi)(d)| + |( S_{Ne_n} \phi)(\sigma^n(d))|
\\
&
\le
3Ne_nW + Ne_n \nbar \phi \nbar_\mathsf{u}\\
&
\le
4Ne_nW
\end{aligned}
\end{equation}
by trivially bounding the last term.

Choose by Lemma~\ref{lem:sup}, Corollary~\ref{cor:positiveentropy} and Lemma~\ref{lem:dspec} a value of $\epsilon > 0$ such that
\begin{equation}
\label{eqn:pressureboundd}
\frac{1}{n} (S_n \phi)(d) < \press(\phi,\lang) - \epsilon 5NW
\end{equation}
for all large enough $n \in \N$.

Fix $n \in \N$ large enough for \eqref{eqn:pressureboundd} to hold.
We consider separately the cases $e_n < \epsilon n$ and $e_n \ge \epsilon n$.
If $e_n < \epsilon n$ then
\begin{equation}
\label{eqn:dequationone}
\begin{aligned}
\frac{1}{n} (S_n \phi)(b)
&
\le
\frac{1}{n} (S_n\phi)(d) + \frac{4Ne_nW}{n}
\\
&
<
\frac{1}{n} (S_n\phi)(d) + \epsilon 4NW 
\\
&
<
\press(\phi,\lang) - \epsilon NW 
\end{aligned}
\end{equation}
using \eqref{eqn:bdCompare} and \eqref{eqn:pressureboundd}.

Next suppose that $e_n\ge \epsilon n$.
Write $p = \ell + 1$ and choose $0 <\delta < \frac{1}{4}$ so that 
\begin{equation}
\label{eqn:delta}
\log\left(\frac{1}{\delta}\right) > 4N(\log p + 4W)
\end{equation}
holds.
We also suppose $n$ is large enough for
\begin{equation}
\label{eqn:n}
n > \max\left\{\frac{1}{\epsilon\delta},L\right\}
\end{equation}
to hold.
In particular $\epsilon n > \frac{1}{\delta} > 1$ so there exists a positive integer $t \le e_n$ so that 
\begin{equation}
\label{eqn:t}
\epsilon n \le t \le 2\epsilon n
\end{equation}
holds.
Finally, since $\delta t > 1$ we can find $k \in \N$ with
\begin{equation}
\label{eqn:k}
\delta t < k \le 2\delta t
\end{equation}
and $k < \frac{t}{2}$ since $\delta < \frac{1}{4}$.  

Fix any collection of $t$ resets in $b_1^n$ and let $x$ be any sequence in any cylinder set corresponding to a word obtained by replacing some $k$ of these $t$ resets with thier corresponding blocks.
Such a word has length $n+kN$ and we can compare $(S_ \phi)(b)$ and $(S_{n} \phi)(x)$ to get
\begin{equation}
\label{eqn:insertionBowen}
|(S_n \phi)(b) - (S_n\phi)(x)| < 4NkW
\end{equation}
just as we compared $(S_n \phi)(b)$ and $(S_n \phi)(d)$ above.
The words $x_1^{n+kN}$ we get by choosing different subsets of resets of cardinality $k$ are all different.
Indeed, If $x_1^{n+kN}$ and $y_1^{n+kN}$ are words obtained by making different choices of $k$ replacements, then there is an earliest reset in $b_1^n$ that is edited to make one of these words but not the other.
Thus $x_1^{n+kN}$ and $y_1^{n+kN}$ are different.

We worry that the words $x_1^n$ and $y_1^n$ we get by truncating $x_1^{n+kN}$ and $y_1^{n+kN}$ may be the same for different choices of $k$ replacements.
However, since our alphabet has $p$ letters and we are cutting off $kN$ letters, there are at most $p^{kN}$ words of length $n+kN$ that truncate to any particular $x_1^n$.
Thus,
\[
\Lambda_n(\phi,\lang)
\ge
\binom{t}{k}p^{-kN} e^{(S_n\phi)(b) - 4NkW} 
\]
by applying \eqref{eqn:insertionBowen} to all words obtainable by replacing $k$ resets for all $0 \le k \le t$.

Combining with \eqref{eqn:t} and \eqref{eqn:k} gives
\begin{align*}
\log \Lambda_n(\phi,\lang)
&
\ge
t \delta \log \frac{1}{\delta} - 2\log t - kN\log p + (S_n\phi)(b) - 4NkW
\\
&
\ge
\delta \epsilon n \log \frac{1}{\delta} - 2 \log(2\epsilon n) - 4\delta \epsilon n N \log p + (S_n\phi)(b) - 16\delta\epsilon n NW
\end{align*}
after an application of Lemma~\ref{lem:logbinomial}.
We conclude that
\begin{equation}
\label{eqn:dequationtwo}
\frac{1}{n} \log \Lambda_n(\phi,\lang) \ge \frac{1}{n} (S_n\phi)(b) + \delta \epsilon \left( \log \frac{1}{\delta} - 4N(\log p + 4W) \right) - \frac{2 \log(2\epsilon n)}{n}
\end{equation}
after dividing by $n$ and regrouping.

For all $n$ large enough either \eqref{eqn:dequationone} or \eqref{eqn:dequationtwo} holds.
Therefore
\[
\limsup_{n \to \infty} \frac{1}{n} (S_n\phi)(b)
\le
\press(\phi,\lang) - \epsilon \min\left\{NW, \delta\log\frac{1}{\delta}- 4N\delta(\log p + 4W)  \right\}
\]
which establishes \eqref{eqn:lastGoal} because \eqref{eqn:delta} implies $\delta\log\frac{1}{\delta}- 4N\delta(\log p + 4W)$ is positive.

\section{An example and Theorem~\ref{thm:denseness}}
\label{sec:examples}

In this section we give exemplar $\alpha$ and $\beta$ with the property that the corresponding subshift $\Sigma$ does not have specification but does satisfy the hypothesis of Theorem~\ref{thm:mainTheorem}.
We also prove Theorem~\ref{thm:denseness}.

\subsection{An example}

Throughout this subsection we assume $\alpha \beta = 1$.
This implies $a = 0\overline{1}$ because $F(0) =  \alpha \bmod 1$ and $F^2(0) = \beta \alpha + \alpha \bmod 1 = \alpha \bmod 1$ satisfies
\[
\frac{1-\alpha}{\beta} \le \alpha < \frac{2 - \alpha}{\beta}
\]
via $\alpha \beta = 1$ and $0 \le \alpha < 1$.

Our goal is to choose $\beta > 1$ such the sequence
\begin{equation}
\label{eqn:goodSequence}
c = 32 012 0112 01112 011112 \cdots 0 1^n 2 \cdots = 32 \prod_{n=1}^\infty 01^n 2
\end{equation}
is equal to $b$.
It is natural to attempt to choose $\beta$ according to the relationship
\begin{equation}
\label{eqn:exampleBeta}
1 = \frac{c_1 - \frac{1}{\beta}}{\beta} + \frac{c_2 - \frac{1}{\beta}}{\beta^2} + \frac{c_3 - \frac{1}{\beta}}{\beta^3} + \frac{c_4 - \frac{1}{\beta}}{\beta^4} + \cdots
\end{equation}
but \eqref{eqn:exampleBeta} does not automatically imply that $b$ will have the desired form.
To see this, write $\eta$ for the map
\[
\eta(d) = \frac{d_1 - \frac{1}{\beta}}{\beta} + \frac{d_2 - \frac{1}{\beta}}{\beta^2} + \frac{d_3 - \frac{1}{\beta}}{\beta^3} + \frac{d_4 - \frac{1}{\beta}}{\beta^4} + \cdots
\]
from $\{0,\dots,\ell\}^\N \to \R$.

\begin{example}
With $\beta = 3$ we have
\[
1 = \frac{2 - \frac{1}{3}}{3} + \frac{3 - \frac{1}{3}}{3^2} + \frac{3 - \frac{1}{3}}{3^3} + \frac{3 - \frac{1}{3}}{3^4} + \cdots
\]
and yet $b = 3 \overline{1}$ for this $\beta$.
\end{example}

\begin{lemma}
\label{lem:10positive}
If $\beta > 2$ then $\eta(1 \overline{0}) > 0$.
\end{lemma}
\begin{proof}
One calculates that
\[
\eta(1 \overline{0}) = \frac{\beta^2 - 2\beta}{\beta^2(\beta - 1)}
\]
which is certainly positive when $\beta > 2$.
\end{proof}

\begin{lemma}
\label{lem:negativeCrit}
If $\beta > 2$ and $\eta(d) < 0$ then $d_1 = 0$.
\end{lemma}
\begin{proof}
From $\eta(d) < 0$ we deduce
\[
\frac{d_1 - \frac{1}{\beta}}{\beta} - \left( \frac{1}{\beta^3} + \frac{1}{\beta^4} + \cdots \right)
<
0
\]
since all $x_n \ge 0$.
This gives $x_1 < \frac{1}{\beta} ( 1 + \frac{1}{\beta} )$ forcing $x_1 = 0$.
\end{proof}

\begin{lemma}
If $\beta > 3$ then $\eta(2 \overline{3}) < 1$.
\end{lemma}
\begin{proof}
One calculates that $\eta(2 \overline{3}) = \frac{2}{\beta - 1}$ which is at most 1 when $\beta > 3$.
\end{proof}

\begin{lemma}
If $3 < \beta < 3.73$ and $\eta(d) > 1$ then $d_1 = 3$.
\end{lemma}
\begin{proof}
If $1 < \eta(d)$ then
\[
1 < \frac{d_1 - \frac{1}{\beta}}{\beta} + \frac{3 - \frac{1}{\beta}}{\beta^2} \frac{1}{1 - \frac{1}{\beta}}
\]
because all $d_n$ are at most $3$.
\end{proof}

\begin{lemma}
If $\beta > 2$ then $\eta(2 \overline{0}) > \eta(\overline{1})$.
\end{lemma}
\begin{proof}
Geometric series calculations give $\eta(2 \overline{0}) - \eta(\overline{1}) = \eta(1 \overline{0})$ which is positive by Lemma~\ref{lem:10positive}.
\end{proof}

Fix now $\beta$ satisfying \eqref{eqn:exampleBeta}.
Such a $\beta$ exists by the intermediate value theorem.
Indeed, the right-hand side of \eqref{eqn:exampleBeta} is at most
\[
\eta(3\overline{2}) = \frac{3}{\beta} + \frac{1}{\beta(\beta - 1)}
\]
and, from
\[
\frac{2}{\beta^j} > \frac{1}{\beta^j} + \frac{1}{\beta^{j+1}}
\]
is at least $\eta(3\overline{1}) = \frac{3}{\beta}$.
Morevoer, these inequalities imply that $3 < \beta < 3.73$ whence $3 < \beta + \alpha < 4$.
We claim for this value of $\beta$ that $b = c$.
Certainly $b_1 = 3$.
Since $b_n = \floor{T^n(1)}$ it suffices to prove that $0 \le \eta(\sigma^n c) < 1$ for all $n \in \N$.
By Lemma~\ref{lem:negativeCrit} this can only happen for those $n$ at which $c_n = 0$.
But in all such cases
\[
\eta(\sigma^n c)
\ge
\frac{0 - \frac{1}{\beta}}{\beta} + \frac{1 - \frac{1}{\beta}}{\beta} \frac{1}{1 - \frac{1}{\beta}}
=
\frac{1}{\beta} - \frac{1}{\beta^2} > 0
\]
so we always have $\eta(\sigma^n c) > 0$.
Also
\[
\eta(\sigma^n c)
\le
\frac{2 - \frac{1}{\beta}}{\beta} + \frac{2 - \frac{1}{\beta}}{\beta^2} + \frac{3 - \frac{1}{\beta}}{\beta^3} + \frac{4 - \frac{1}{\beta}}{\beta^4} + \cdots
=
\frac{2 \beta - 1}{\beta(\beta - 1)}
<
1
\]
when $3 < \beta < 4$ so $\eta(\sigma^n c) < 1$ for all $n \in \N$.

This implies $b = c$ with $c$ as in \eqref{eqn:goodSequence}.
Since $\alpha \beta = 1$ we also have $a = 0 \overline{1}$.
In this case the set $\diag(b)$ is unbounded so Theorem~\ref{thm:specChar} implies the subshift $\Sigma$ does not have specification.
However, the set $\diag(a)$ is bounded, so the subshift does have unique equilibrium states for \Holder{} potentials via Theorem~\ref{thm:mainTheorem}.

\subsection{Proof of Theorem~\ref{thm:denseness}}

\begin{proof}
[Proof of Theorem~\ref{thm:denseness}]
For any $\alpha$ and any $\beta$, the existence of a bound on the lengths of initial subwords of $b$ that appear in $a$ is equivalent to the existence of a bound on the distance between one and points in the orbit of zero.
There will certainly be such a bound when the orbit of $0$ is periodic.

Suppose $\beta,\alpha$ are chosen so that $F^n(0)$ gets arbitrarily close to $1$.
If
\[
F^n(0)=1-\delta
\] 
for $\delta > 0$ small then 
\[
F^{n-1}(0) = q_j - \frac{\delta}{\beta}
\]
for some $j \in \{0,\dots,\ell\}$.
That is, if some point in the orbit of $0$ is very close to $1$, then the previous point in the orbit is even closer to a point of discontinuity of $F$ and lies to the left of the discontinuity.
Let $m$ be a time such that $F^m(0)$ is closer to $1$ then ever before.
Then $F^{m-1}(0)$ is closer to a discontinuity of $F$ then ever before while lying to the left of the discontinuity.
Consider $\alpha' = \alpha + \epsilon$ for $\epsilon$ small, and the corresponding transformation $F'$.
Observe that the discontinuities of $F'$ are $\epsilon/\beta$ to the left of the discontinuities of $F$, and for $n \in \{1,2,\dots,m-1\}$ the points $(F')^n(0)$ are $\epsilon(1+\beta+ \beta^2 + \cdots + \beta^{n-1})$ to the right of $F^n(0)$.
Thus, there is a particular $\epsilon$, depending on $m$ and how close $F^m(0)$ is to $1$, such that $(F')^{m-1}(0)$ is equal to a point of discontinuity of $F'$.
This implies $(F')^m(0) = 0$ and hence that the orbit of $0$ is periodic.
By choosing $m$ arbitrarily large, we get $\alpha'$ arbitrarily close to $\alpha$. 
\end{proof}

\printbibliography

\end{document}